\documentclass[12pt,a4paper]{article}

\usepackage{amsmath, amsthm, amsfonts, amssymb,color,geometry,enumerate}
\usepackage{graphicx}
\usepackage{fancybox}
\usepackage{cases}
\usepackage{fancyhdr}
\usepackage{hyperref} 

\theoremstyle{definition}

\newtheorem{theorem}{Theorem}[section]
\newtheorem{proposition}[theorem]{Proposition}
\newtheorem{lemma}[theorem]{Lemma}
\newtheorem{definition}[theorem]{Definition}
\newtheorem{remark}[theorem]{Remark}
\newtheorem{problem}[theorem]{Problem}
\newtheorem{example}[theorem]{Example}
\newtheorem{conjecture}[theorem]{Conjecture}
\newtheorem{corollary}[theorem]{Corollary}

\makeatletter

\@addtoreset{equation}{section}
\makeatother

\newcommand{\R}{\mathbb{R}}   
\renewcommand{\epsilon}{\varepsilon}    
\newcommand{\supp}{{\rm supp}}   
\newcommand{\Ber}{{\mathbf b}}  
\newcommand{\PK}{{\mathbf p}}  

\begin{document}

\title{{\bf \Large{Unimodality for free multiplicative convolution with free normal distributions on the unit circle}}}
\author{\Large{Takahiro Hasebe and Yuki Ueda}}
\date{}

\maketitle

\begin{abstract}
We study unimodality for free multiplicative convolution with free normal distributions $\{\lambda_t\}_{t>0}$ on the unit circle. We give four results on unimodality for $\mu\boxtimes\lambda_t$: (1) if $\mu$ is a symmetric unimodal distribution on the unit circle then so is $\mu\boxtimes \lambda_t$ at any time $t>0$; (2) if $\mu$ is a symmetric distribution on $\mathbb{T}$ supported on $\{e^{i\theta}: \theta \in [-\varphi,\varphi]\}$ for some $\varphi \in (0,\pi/2)$, then $\mu \boxtimes \lambda_t$ is unimodal for sufficiently large $t>0$; (3) $\Ber \boxtimes \lambda_t$ is not unimodal at any time $t>0$, where $\Ber$ is the equally weighted Bernoulli distribution on $\{1,-1\}$; (4) $\lambda_t$ is not freely strongly unimodal for sufficiently small $t>0$. Moreover, we study unimodality for classical multiplicative convolution (with Poisson kernels), which is useful in proving the above four results. 
\end{abstract}



\section{Introduction}
In free probability, the semicircle distribution 
\begin{equation}
S(m,v)(dx):=\frac{1}{2\pi v} \sqrt{4v-(x-m)^2} \cdot 1_{[m-2\sqrt{v}, m+2\sqrt{v}]}(x)\,dx, \qquad x\in\mathbb{R}
\end{equation} 
with mean $m \in \mathbb R$ and variance $v>0$ plays the role of the normal distribution in probability theory. 
The distribution of the sum of two free random variables which follow $\mu$ and $\nu$ respectively is denoted by $\mu \boxplus \nu$. The operation $\boxplus $ is called {\it free additive convolution} (see \cite{BV93}).  The probability measure $\mu \boxplus S(0,t)$ is of particular importance since it describes the law of free Brownian motion with initial distribution $\mu$. In \cite{Bia97(3)}, Biane studied regularity of $\mu \boxplus S(0,t)$, in particular, gave a density formula of $\mu\boxplus S(0,t)$.

One notion that captures a visual aspect of a probability measure is {\it unimodality}. 
By using Biane's density formula, the authors studied unimodality for $\mu \boxplus S(0,t)$ in \cite{HU18} and proved the following.  
\begin{enumerate}[(1)]
\item {\bf Unimodality for $\mu \boxplus S(0,t)$}: let $\mu$ be a symmetric unimodal probability measure on $\mathbb{R}$. Then $\mu \boxplus S(0,t)$ is unimodal at any time $t>0$.
\item {\bf Eventual unimodality for $\mu \boxplus S(0,t)$}: let $\mu$ be a compactly supported probability measure on $\mathbb{R}$. Then $\mu \boxplus S(0,t)$ is unimodal for sufficiently large $t>0$.
\item {\bf Non-unimodality for $\mu \boxplus S(0,t)$}: there exists a (non-compactly supported) probability measure $\mu$ on $\mathbb{R}$, such that $\mu\boxplus S(0,t)$ is not unimodal at any time $t>0$.
\item {\bf Failure of freely strong unimodality of $S(0,1)$}: there exists a unimodal probability measure $\mu$ on $\mathbb{R}$, such that $\mu\boxplus S(0,1)$ is not unimodal.
\end{enumerate}
The first three statements have counterparts in classical probability, namely the statements hold true for $\mu \ast N(0,t)$ instead of $\mu \boxplus S(0,t)$. On the other hand, the fourth statement is not the case in classical probability, that is, $\mu \ast N(0,t)$ is unimodal for any unimodal probability measure $\mu$ and any $t>0$. 
Thus strong unimodality does not show complete similarity between the classical and free probability theories. 

Other similarities/dissimilarities on unimodality were found in the literature. Yamazato proved a remarkable property that every selfdecomposable distribution (in particular stable distributions) is unimodal (see \cite{Yam78}). More generally, Yamazato and Wolfe studied unimodality of infinitely divisible distributions (e.g.\ \cite{Wol78, Yam82}). In free probability, Biane firstly proved that every freely stable law is unimodal (see \cite{BP99}) and then Hasebe and Thorbj\o rnsen proved that every freely selfdecomposable distribution is unimodal (see \cite{HT}). This is a complete free analogue of Yamazato's result. However, Hasebe and Sakuma pointed out dissimilarities between classical and free probability theories regarding unimodality for general infinitely divisible distributions (see \cite{HS17}).

In this paper, we study unimodality for free multiplicative convolution with the free normal distributions on the unit circle. We write $\lambda_t$ $(t>0)$ for the free normal distribution on the unit circle which is the distribution of free unitary Brownian motion at time $t>0$ (see \cite{Bia97, Bia98}). Let $U_t\sim \lambda_t$ ($t>0$) and $V\sim \mu$ be freely independent unitary operators in a noncommutative probability space. Then we write $\mu\boxtimes \lambda_t$ for the distribution of $VU_t$ for $t>0$. The operation $\boxtimes$ is called {\it free multiplicative convolution}. It can be defined on the set of general probability measures on $\mathbb{T}$. In \cite{Z15}, Zhong proved that for an arbitrary probability measure $\mu$ on $\mathbb{T}$ the probability measure $\mu\boxtimes \lambda_t$ is Haar absolutely continuous and its probability density function is continuous on $\mathbb T$ and is analytic wherever it is positive. In this paper, we study the unimodality for $\mu\boxtimes \lambda_t$ by using Zhong's density formula and we conclude the following results:

\begin{theorem}\label{U1} {\it (See Theorems \ref{main-t}, \ref{main-t1}, \ref{main-t2} and \ref{main-t3} below.)}
\begin{enumerate}[(1)]

\item\label{item:unimodality} {\bf Unimodality for $\mu\boxtimes \lambda_t$}: let $\mu$ be a symmetric unimodal probability measure on $\mathbb{T}$. Then $\mu\boxtimes \lambda_t$ is symmetric unimodal for all $t>0$.

\item \label{item:eventual}{\bf Eventual unimodality for $\mu\boxtimes \lambda_t$}:  let $\mu$ be a symmetric distribution on $\mathbb{T}$ supported on $\{e^{i\theta}: \theta \in [-\varphi,\varphi]\}$ for some $\varphi \in (0,\pi/2)$. Then $\mu \boxtimes \lambda_t$ is unimodal for all $t\geq \frac{2(1+r_\varphi)}{1-r_\varphi}\log\frac{1}{r_\varphi}$, where $r_\varphi=(\cos\varphi)/4$.

\item\label{item:non-unimodality} {\bf Non-unimodality for $\mu\boxtimes \lambda_t$}: let $\Ber$ be the Bernoulli distribution $\frac{1}{2}(\delta_{-1}+\delta_1)$ on $\mathbb{T}$. Then the density of $\Ber \boxtimes \lambda_t$ attains a strict maximum at $\pm1$ and hence is not unimodal for any $t>0$.

\item {\bf Failure of freely strong unimodality of $\lambda_t$}: there exists some $t_0>0$ such that $\lambda_{t}$ is not freely strongly unimodal at any time $t \in (0,t_0)$.
\end{enumerate}
\end{theorem}
The above \eqref{item:unimodality} strengthens \cite[Corollary 3.13]{Z15}, where the statement is proved for $\mu =\delta_1$. The non-unimodality result \eqref{item:non-unimodality} contrasts with the additive case; in the latter case the assumption of compact support of the initial distribution implied the eventual unimodality. We do not know whether the failure of freely strong unimodality can be proved for all $t>0$. 

We also prove similar statements for classical multiplicative convolution $\circledast$ on the unit circle, in particular convolution with the Poisson kernel $\PK_{r,\psi}$ with the density 
\begin{equation}\label{eq:PK}
\frac{d\PK_{r,\psi}}{d\theta}(\theta) =\frac{1-r^2}{2\pi |1-re^{i(\theta-\psi)}|^2}, \qquad \theta\in[-\pi,\pi), 
\end{equation}
where $r\in [0,1)$ and $\psi\in\mathbb{R}$. 

\begin{theorem}\label{U2} {\it (See Theorems \ref{thm:classical_symmetric_unimodal}, \ref{thm:eventual unimodality of PK}, \ref{thm:b circledast PK} and \ref{s-unim} below.)} 
\begin{enumerate}[(1)]
\item \label{item:symmetric}
Let $\mu$ and $\nu$ be symmetric unimodal probability measures on $\mathbb T$. Then so is $\mu \circledast \nu$. 
\item {\bf Eventual unimodality for $\PK_{r,\psi}\circledast \mu$}: let $\mu$ be a symmetric distribution on $\mathbb{T}$ supported on $\{e^{i\theta}: \theta \in [-\varphi,\varphi]\}$ for some $\varphi \in (0,\pi/2)$. Then $\PK_{r,\psi} \circledast \mu$ is unimodal for all $r\in (0,r_\varphi)$, where $r_\varphi=(\cos\varphi)/4$.

\item {\bf Non-unimdality for $\PK_{r,\psi}\circledast \mu$}: the measure $\PK_{r,\psi}\circledast\Ber$ attains a strict maximum at $\pm1 \in \mathbb T$ and hence is not unimodal at any $r\in(0,1)$.
\item \label{item:strongly unimodal} {\bf Failure of strong unimodality of $\PK_{r,\psi}$}: 
there exists some $r_0\in(0,1)$ such that $\PK_{r,\psi}$ is not strongly unimodal for any $r\in(r_0,1)$ and $\psi \in\R$. 
\end{enumerate}
\end{theorem} 
In fact, some statements of Theorem \ref{U1} are closely related to or directly applied to Theorem \ref{U2}, 
 as well as they are new in classical probability to the authors' knowledge. The reason why classical convolution with Poisson kernels is related to free convolution with free normal distributions can be understood via Lemma \ref{lem:unimodal} and Lemma \ref{lem:derivative_v}. 

After introducing basic known results on free probability in Section \ref{sec2}, 
we introduce and investigate the classes of symmetric and unimodal probability measures on $\mathbb{T}$ in Sections \ref{sec3} and \ref{sec4}, respectively. The above main results are proved in Section \ref{sec5}.


\section{Preliminaries in free probability theory}\label{sec2}

\subsection{$\Sigma$-transform and free multiplicative convolution}\label{subsec1}
Throughout the paper we identify the unit circle $\mathbb T$ with $\mathbb R / 2\pi \mathbb Z$ (often denoted as $[-\pi,\pi)$) and use the notation $\mu(d\theta)$ as well as $\mu(d\xi)$, where $\xi=e^{i\theta}$, for probability measures $\mu$ on $\mathbb T.$ Moreover, we identify the density functions of a Haar absolutely continuous distributions on $\mathbb{T}$ with functions on $\mathbb{R}$ with period $2\pi$.

Free multiplicative convolution  $\boxtimes$  is a binary operation on the set of probability measures on $\mathbb{T}$. It was  introduced in \cite{V87} as the distribution of the multiplication of free unitary random variables on a noncommutative probability space. We briefly review on how to compute it, following \cite{BV92}. 

For a probability measure $\mu$ on $\mathbb{T}$, we define the analytic functions 
\begin{equation}
\psi_\mu(z):=\int_\mathbb{T} \frac{\xi z}{1-\xi z}\mu(d\xi) \text{~~~and~~~} \eta_\mu(z):=\frac{\psi_\mu(z)}{1+\psi_\mu(z)}, \qquad z\in \mathbb{D}.
\end{equation}
The function $\eta_\mu$ is called the {\it $\eta$-transform of $\mu$}. It is easy to check that $\eta_\mu(0)=0$ and $\eta_\mu'(0)=\int_\mathbb{T} \xi \,d\mu(\xi)$. If the first moment of $\mu$ is nonzero then there exists a function $\eta_\mu^{-1}(z)$ which is analytic in a neighborhood of zero, such that
\begin{equation}
\eta_\mu(\eta_\mu^{-1}(z))=\eta_\mu^{-1}(\eta_\mu(z))=z,
\end{equation}
for $z$ in a neighborhood of zero. Then we define the {\it $\Sigma$-transform of $\mu$}
\begin{equation}
\Sigma_\mu(z):=\frac{\eta_\mu^{-1}(z)}{z}  
\end{equation}
in the neighborhood of zero where $\eta_\mu^{-1}(z)$ is defined. 
For any probability measures $\mu, \nu$ on $\mathbb{T}$ with nonzero first moments, we can find a unique probability measure $\tau$ on $\mathbb{T}$ with a nonzero first moment such that 
\begin{equation}
\Sigma_{\tau}(z)=\Sigma_\mu(z)\Sigma_\nu(z),
\end{equation}
for $z$ in a neighborhood of zero where the three functions are defined. We call the probability measure $\tau$ the {\it free multiplicative convolution} of $\mu$ and $\nu$, and it is denoted by $\mu\boxtimes \nu$. The $\Sigma$-transform is related to the {\it S-transform} via 
\begin{equation}
S_\mu(z)=\Sigma_\mu\left(\frac{z}{1+z}\right)=\frac{1+z}{z}\psi_\mu^{-1}(z).
\end{equation}
For all probability measures $\mu,\nu$ on $\mathbb{T}$ with nonzero first moments, there exists an analytic map $\eta:\mathbb{D}\rightarrow \mathbb{D}$ such that $\eta(0)=0$ and $\eta_{\mu\boxtimes \nu}(z)=\eta_\mu(\eta(z))$ holds for all $z\in\mathbb{D}$. The function $\eta$ is called the {\it subordination function of $\eta_{\mu\boxtimes\nu}$ with respect to $\eta_\mu$}. 

If both $\mu$ and $\nu$ have the zero first moment, the definition of free independence implies that $\mu\boxtimes \nu={\bf h}$, where the measure ${\bf h}$ is the normalized Haar measure.

\subsection{
Free normal distributions on the unit circle} 
A probability measure $\mu$ on $\mathbb{T}$ is said to be {\it $\boxtimes$-infinitely divisible} if for each $n\in\mathbb{N}$ there exists a probability measure $\mu$ on $\mathbb{T}$, such that
\begin{equation}
\mu=\overbrace{\mu_n\boxtimes \cdots \boxtimes \mu_n}^{n \text{ times}}.
\end{equation}
In \cite[Theorem 6.7]{BV92}, $\mu$ is $\boxtimes$-infinitely divisible on $\mathbb{T}$ if and only if there exists a function $u(z)$ which is analytic in $\mathbb{D}$ such that $\Re u(z)\ge0$ for $z\in\mathbb{D}$ and its $\Sigma$-transform is expressed by
\begin{equation}\label{S}
\Sigma_\mu(z)=\exp [u(z)].
\end{equation}
Suppose that $\mu$ is $\boxtimes$-infinitely divisible on $\mathbb{T}$ and its $\Sigma$-transform is expressed by \eqref{S}. Then there exists a $\boxtimes$-semigroup $\{\mu_t\}_{t\ge0}$ of probability measures on $\mathbb{T}$ such that
\begin{equation}
\Sigma_{\mu_t}(z)=\exp [tu(z)], \qquad t\geq0, 
\end{equation}
in particular $\mu_0=\delta_1$. 
For $t>0$, we define $\lambda_t$ as the $\boxtimes$-infinitely divisible probability measure on $\mathbb{T}$ whose $\Sigma$-transform is given by 
\begin{equation}
\Sigma_{\lambda_t}(z)=\exp \left[\frac{t}{2}\cdot \frac{1+z}{1-z} \right], \qquad t>0.
\end{equation}
We call $\lambda_t$ the {\it free normal distribution on the unit circle}. This is the distribution of {\it free unitary Brownian motion} which was introduced by Biane (e.g, see \cite{Bia97,Bia97(2)}). The measure $\lambda_t$ has the first moment $e^{-t}$. 
\subsection{Zhong's density formula}
Most of the materials in this section are based on Zhong's papers \cite{Z14,Z15}.  As mentioned in Section \ref{subsec1}, if $\mu$ and $\nu$ have nonzero first moments, then we can find a subordination function of $\eta_{\mu\boxtimes\nu}$ with respect to $\eta_\mu$ (or $\eta_\nu$). However, if $\mu$ has the zero first moment (and $\nu$ has a nonzero first moment), then the subordination function of $\eta_{\mu\boxtimes \nu}$ with respect to $\eta_\mu$ is generally not unique (see \cite[Example 3.5]{Z14}). If we require a subordination function to satisfy additional properties, then it is unique even if $\mu$ has the zero first moment \cite{BB07,Bia98}. Namely, if $\mu\neq {\bf h}$ is a probability measure on $\mathbb{T}$ and $\nu$ has a nonzero first moment, then there exists a unique pair of analytic functions $w_1,w_2:\mathbb{D}\rightarrow \mathbb{D}$ such that
\begin{itemize}
\item  $w_1(0)=w_2(0)=0$;
\item  $\eta_{\mu\boxtimes \nu }(z)=\eta_\mu(w_1(z))=\eta_\nu(w_2(z))$ for all $z\in\mathbb{D}$;
\item $w_1(z)w_2(z)=z\eta_{\mu\boxtimes \nu}(z)$ for all $z\in\mathbb{D}$.
\end{itemize}
By the above result, for a probability measure $\mu\neq{\bf h}$ on $\mathbb{T}$ and $t>0$, we can find a unique pair of  subordination functions $\eta_t,\zeta_t:\mathbb{D}\rightarrow \mathbb{D}$ of $\eta_{\mu\boxtimes \lambda_t}$ with respect to $\eta_{\mu}$ and $\eta_{\lambda_t}$, respectively such that
\begin{itemize}
\item $\eta_t(0)=\zeta_t(0)=0$,
\item $\eta_{\mu\boxtimes \lambda_t}(z)=\eta_\mu(\eta_t(z))=\eta_{\lambda_t}(\zeta_t(z))$,
\item $\eta_t(z)\zeta_t(z)=z\eta_{\mu\boxtimes \lambda_t}(z)$,
\end{itemize}
for all $z\in \mathbb{D}$. According to \cite[Proposition 3.2]{BB05}, for each $t>0$ there exists a probability measure $\rho_t$ on $\mathbb{T}$, such that $\eta_t(z)=\eta_{\rho_t}(z)$. In \cite[Lemma 3.4 and Corollary 3.13]{Z14}, the measures $\rho_t$ are $\boxtimes$-infinitely divisible on $\mathbb{T}$ and its $\Sigma$-transform is expressed by
\begin{equation}
\Sigma_{\rho_t}(z)=\Sigma_{\lambda_t}(\eta_\mu(z))=\exp \left[ \frac{t}{2} \int_\mathbb{T} \frac{1+\xi z}{1-\xi z} \mu(d\xi) \right].
\end{equation}
It is proved in \cite{BB05} that $\eta_t(z)=\eta_{\rho_t}(z)$ is a conformal map of $\mathbb D$ onto a simply connected domain $\Omega_{t,\mu} \subset \mathbb D$ and extends to a homeomorphism from $\text{cl}(\mathbb D)$ to $\text{cl}(\Omega_{t,\mu})$, where $\text{cl}$ means the closure operation. To study the measure $\mu\boxtimes \lambda_t$, it is important to describe the domain $\Omega_{t,\mu}$. 

According to Zhong's paper (see \cite{Z15}), we define the open set (of $\mathbb T$)
\begin{equation}
U_{t,\mu}:=\left\{\theta\in [-\pi,\pi): \int_{-\pi}^\pi \frac{1}{|1-e^{i (\theta+x)}|^2}\mu(dx)>\frac{1}{t} \right\},
\end{equation}
and also define a function $v_{t,\mu}:[-\pi,\pi)\rightarrow (0,1]$ as
\begin{equation}
v_{t,\mu}(\theta):=\sup \left\{ r\in (0,1) : \frac{1-r^2}{-2 \log r}\int_{-\pi}^\pi \frac{1}{|1-re^{i(\theta+x)}|^2}\mu(dx)<\frac{1}{t} \right\}.
\end{equation}
Note that the function $v_{t,\mu}$ is continuous on $[-\pi,\pi)$, analytic in $U_{t,\mu}$ and one has $U_{t,\mu}=\{ \theta\in [-\pi,\pi) : v_{t,\mu}(\theta)\neq 1\}$ by the proof of \cite[Corollary 3.9]{Z15}. In particular we have that $v_{t,\mu}(\theta)=v_{t,\mu}(-\theta)$ for all $\theta\in[-\pi,\pi)$ if $\mu$ is symmetric. For any $\theta\in U_{t,\mu}$, the value $v_{t,\mu}(\theta)$ is a unique solution $r \in (0,1)$ of the following equation:
\begin{equation}\label{eq:v}
\frac{1-r^2}{-2\log r}\int_{-\pi}^\pi \frac{1}{|1-re^{i(\theta+x)}|^2}\mu(dx)=\frac{1}{t}.
\end{equation}

In \cite[Theorem 3.2]{Z15}, it is proved that $\Omega_{t,\mu}=\{re^{i\theta}: 0 \leq r < v_{t,\mu}(\theta), \theta \in [-\pi,\pi)\}$ and  
$\partial \Omega_{t,\mu}\cap \mathbb{D} =\{v_{t,\mu}(\theta)e^{i\theta}:\theta \in U_{t,\mu}\}$. 
As the inverse map of $\eta_t$, the restriction of the map $\Phi_{t,\mu}(z) := z \Sigma_{\rho_t}(z)$ to $\Omega_{t,\mu}$ is conformal onto $\mathbb D$, and it extends to a homeomorphism of $\text{cl}(\Omega_{t,\mu})$ onto $\text{cl}(\mathbb D)$. We define a homeomorphism $\Psi_{t,\mu}:\mathbb{T}\rightarrow \mathbb{T}$ by
\begin{equation}
\Psi_{t,\mu}(e^{i\theta}):=\Phi_{t,\mu}(v_{t,\mu}(\theta)e^{i\theta}), \qquad \theta\in [-\pi,\pi).
\end{equation}
Since $|\Psi_{t,\mu}(e^{i\theta})|=1$ and $\Phi_{t,\mu}(z)= z \Sigma_{\rho_t}(z)$, we have 
\begin{equation}
\arg (\Psi_{t,\mu}(e^{i\theta}))=\theta+t\int_{-\pi}^\pi \frac{v_{t,\mu}(\theta)\sin(\theta+x)}{|1-v_{t,\mu}(\theta)e^{i(\theta+x)}|^2}\mu(dx),
\end{equation}
and hence 
\begin{equation}
\Psi_{t,\mu}(e^{i\theta})=\exp \left[i\left(\theta+t\int_{-\pi}^\pi \frac{v_{t,\mu}(\theta)\sin(\theta+x)}{|1-v_{t,\mu}(\theta)e^{i(\theta+x)}|^2}\mu(dx) \right)\right].
\end{equation}
By \cite[Proposition 3.6, Theorem 3.8]{Z15}, the probability measure $\mu\boxtimes \lambda_t$ is Haar absolutely continuous and its probability density function $p_t$ is given by
\begin{equation}\label{eq:density}
p_t(\overline{\Psi_{t,\mu}(e^{i\theta})})=-\frac{\log(v_{t,\mu}(\theta))}{\pi t},
\end{equation}
and it is analytic wherever it is positive. Moreover we have that 
\begin{equation}
\supp(\mu\boxtimes\lambda_t)=\{\overline{\Psi_{t,\mu}(e^{i\theta})}: \theta\in \text{cl}(U_{t,\mu})\}. 
\end{equation}

 \subsection{Free multiplicative convolution with Poisson kernels}
There is no useful formula for describing the absolutely continuous part of free multiplicative convolution of general probability measures. One special case is the free multiplicative convolution with the free normal distribution, which has some implicit formula for the density as already mentioned. 
Another computable case is the free multiplicative convolution with the Poisson kernel. 

\begin{proposition}\label{prop:Multiplicative with Poisson}
For every probability measure $\mu$, every $r \in [0,1)$ and every $\psi \in \R$ we have
\begin{equation}
\PK_{r,\psi} \circledast \mu = \PK_{r,\psi} \boxtimes \mu. 
\end{equation}
\end{proposition}
 A proof can be found e.g.\ in \cite[Section 7.1]{FHS18}. 
In Sections \ref{sec5.2} -- \ref{sec5.3}, we will discuss unimodality for classical multiplicative convolution (and hence free multiplicative convolution) with Poisson kernels.

\section{Symmetric probability measures on the unit circle}\label{sec3}

In this section, we define and characterize the symmetric probability measures on $\mathbb{T}$.

\begin{definition}
A probability measure $\mu$ on $\mathbb{T}$ is said to be {\it symmetric} if $\mu(B)=\mu(\overline{B})$ for all Borel sets $B$ in $\mathbb{T}$, where $\overline{B}:=\{\overline{\xi}: \xi\in B\}$.
\end{definition}

We show that the free multiplicative convolution of two symmetric probability measures on $\mathbb{T}$ is also symmetric. To prove this we characterize the class of symmetric probability measures on $\mathbb{T}$. 

\begin{proposition}\label{Symm}
For a probability measure $\mu$ on $\mathbb{T}$, the following statements are equivalent. 
\begin{enumerate}[(1)]
\item\label{item:symmetric1} $\mu$ is symmetric.
\item\label{item:symmetric2} $\psi_\mu(\overline{z})=\overline{\psi_\mu(z)}$ for $z\in\mathbb{D}$.
\item\label{item:symmetric3} All moments of $\mu$ are real. 
\end{enumerate}
Moreover, if $\mu$ has a non-zero mean, then the above conditions are also equivalent to 
\begin{enumerate}[(1)]\setcounter{enumi}{3}
\item\label{item:symmetric4} $\Sigma_\mu(\overline{z})=\overline{\Sigma_\mu(z)}$ on a neighborhood of $0$. 
\end{enumerate}
\end{proposition}
\begin{proof}
\eqref{item:symmetric1}$\Rightarrow$\eqref{item:symmetric2}: for a symmetric probability measure $\mu$ on $\mathbb{T}$ we have
\begin{equation*}
\begin{split}
\psi_\mu(z)&=\int_{-\pi}^\pi \frac{\xi z}{1-\xi z}\mu(d\xi)=\int_{(0,\pi)} \left(\frac{\overline{\xi} z}{1-\overline{\xi}z}+\frac{\xi z}{1-\xi z}\right)\mu(d\xi) + \frac{z}{1-z} \mu(\{1\}) + \frac{-z}{1+z} \mu(\{-1\}), 
\end{split}
\end{equation*}
and therefore we have $\psi_\mu(\overline{z})=\overline{\psi_\mu(z)}.$

\eqref{item:symmetric2}$\Rightarrow$\eqref{item:symmetric1}: condition \eqref{item:symmetric2} implies that $\Re\psi_\mu(z) = \Re \psi_\mu(\overline{z})$. The inversion formula \cite[Theorem 2]{K59} shows that 
$$
\mu(A) = \lim_{r\uparrow1}\int_A (2 \Re(\psi_\mu(r\overline{\xi}))+1) \,{\bf h}(d\xi) =\lim_{r\uparrow1}\int_A (2 \Re(\psi_\mu(r\xi))+1) \,{\bf h}(d\xi) =\mu(\overline{A})
$$
for any arc $A \subset \mathbb T$ such that its endpoints are continuity points of $\mu$. Therefore $\mu(B) = \mu(\overline{B})$ holds for all Borel subsets $B$ of $\mathbb T$. 

Conditions \eqref{item:symmetric2}, \eqref{item:symmetric3} and \eqref{item:symmetric4} are equivalent by basic complex analysis. 
\end{proof}

\begin{corollary}\label{Ssym}
If $\mu$ and $\nu$ are symmetric probability measures on $\mathbb{T}$, then $\mu\boxtimes \nu$ and $\mu \circledast \nu$ are also symmetric.
\end{corollary}
\begin{proof}
Any moment of $\mu \boxtimes \nu$ can be expressed as a polynomial (with real coefficients) of moments of $\mu$ and those of $\nu$, and hence is real.  The same arguments apply to $\mu \circledast \nu$. 
\end{proof}
This corollary implies a part of Theorem \ref{U1}: $\mu \boxtimes \lambda_t$ is symmetric whenever $\mu$ is.

\section{Unimodal probability measures on the unit circle}\label{sec4}

In this section, we introduce the concept of unimodal probability measures on $\mathbb{T}$. Recall that a probability measure $\mu$ on $\mathbb{R}$ is said to be {\it unimodal with mode $a\in\mathbb{R}$} if it is written as 
\begin{equation}
\mu(dx)=c\delta_a+f(x)\,dx,
\end{equation}
where $c\geq0$ and the function $f\colon\mathbb{R}\rightarrow[0,\infty)$ is non-decreasing on $(-\infty,a)$ and is non-increasing on $(a,\infty)$. Unimodal distributions can be characterized from the viewpoint of geometric function theory: $\mu$ is unimodal with some mode if and only if the Cauchy-Stieltjes transform 
\begin{equation}
G_\mu(z) =\int_{\R}\frac{1}{z-x} d\mu(x), \qquad \Im(z)>0, 
\end{equation}
is univalent and its range is horizontally convex, namely if $z,w$ are points in the range of $G_\mu$ with the same imaginary part then the line segment connecting them is also contained in the range \cite[Theorem 6.24]{FHS18}. 

According to \cite{FHS18}, we similarly define the concept of unimodal probability measures on $\mathbb{T}$ as follows.

\begin{definition}
Consider $\phi, \psi\in\mathbb{R}$ with $0\le \phi-\psi\le 2\pi$. A probability measure $\mu$ on $\mathbb{T}$ is said to be {\it unimodal with mode $\phi$ and antimode $\psi$} if $\mu$ is written as
\begin{equation}\label{eq:unimodal}
\mu(d\theta)=c\delta_{\phi}+f(\theta)\,d\theta,
\end{equation}
for some nonnegative number $c$ and a function $f\colon (\psi, \psi+2\pi)\rightarrow [0,\infty)$ which is non-decreasing on $(\psi,\phi)$ and non-increasing on $(\phi, \psi+2\pi)$. In particular, if $\psi=\phi$ (resp.\ $\psi+2\pi=\phi$), then we may understand that $f$ is non-increasing (resp.\ non-decreasing) on $(\psi,\psi+2\pi)$. 
\end{definition}

Note that unimodality on $\mathbb T$ can be characterized by the univalence (injectivity) of $\psi_\mu$ and  the "vertical convexity" of the range $\psi_\mu(\mathbb D)$, namely for any $z ,w \in \psi_\mu(\mathbb D)$ having the same real part and for any $t \in(0,1)$, we have $(1-t)z + t w \in \psi_\mu(\mathbb D)$; see \cite[Remark 7.17]{FHS18}. 

\begin{example}\label{wCauchy}
The Poisson kernel (or the wrapped Cauchy distribution) $\PK_{r,\psi}$, defined in \eqref{eq:PK}, is unimodal with mode $\psi$ and antimode $-\pi+\psi$. Moreover, the probability measure
$
p \delta_\psi+(1-p)\PK_{r,\psi}
$
is also unimodal with mode $\psi$ and antimode $-\pi+\psi$ for $p\in(0,1)$.

For other examples, it is proved in \cite[Corollary 3.13]{Z15} that the free normal distribution $\lambda_t$ on the unit circle  is unimodal with mode $0$ and antimode $-\pi$ for all $t>0$. 
\end{example}

It is known that the set of all unimodal probability measures on $\mathbb{R}$ is closed with respect to the weak convergence. A similar result holds for the unit circle by using \cite[Lemma 2.11, Theorem 7.16]{FHS18}. 

\begin{lemma}\label{lem:FHS18}
The set of all unimodal probability measures on $\mathbb{T}$ is closed with respect to the weak convergence.
\end{lemma}

Moreover, unimodal distributions can be approximated by smooth densities. 

\begin{lemma}\label{wclosed}
For a unimodal distribution $\mu$ there is a sequence of probability measures $\{\mu_n\}_{n\geq1}$ which have $C^\infty$ densities such that $\mu_n$ weakly converges to $\mu$ as $n\to\infty$. 
\end{lemma}
\begin{proof}
We only sketch the proof which is similar to \cite[Lemma 6]{HT}. Assume that $\mu$ is of the form \eqref{eq:unimodal}. We may assume that the mode $\phi$ and the antimode $\psi$ are different, and that the antimode $\psi$ is $-\pi$. By approximating $\delta_{\phi}$ with unimodal Haar absolutely continuous distributions supported on a small arc, we may assume that $\mu$ does not have an atom. We regard the density function $f$ as a function on $\R$ supported on $[-\pi,\pi]$. 
By cutting $f$ near $\pm \pi$, flattening $f$ near $\phi$ and normalizing, we may assume that $f(x) =0$ for all $|x|> \pi-\alpha$ for some $\alpha \in(0,\pi)$,  and that $f(x)$ is constant for all $|x-\phi|<\beta$ for some $\beta>0$. Then we apply a mollifier $\varphi_\epsilon := \epsilon^{-1}\varphi(\epsilon^{-1} \cdot)$ with a smooth nonnegative function $\varphi$ supported on $[-1,0]$ to get a smooth function $\varphi_\epsilon \ast f$ which is unimodal on $(\phi,\infty)$, as done in \cite[Lemma 6]{HT}. A similar construction works for the opposite half-line $(-\infty,\phi)$, and combining the two we obtain a unimodal function $f_\epsilon$ on $\R$. Note that if $\epsilon$ is sufficiently small, then $f_\epsilon$ is still supported on $[-\pi,\pi]$, and hence can be regarded as a unimodal density on $\mathbb T$. 
\end{proof}

Unimodality can be characterized in terms of level sets under some assumptions. This idea in the context of free probability was first used in \cite{MR3263926} to prove that a ``free gamma distribution'' is unimodal.

\begin{lemma}\label{lem:basic_unimodal}
Let $\mu$ be a  Haar absolutely continuous distribution on $\mathbb T$. Suppose that its density $p$ extends to a continuous function on $\mathbb T$ and is real analytic in $\{\xi \in \mathbb T: p(\xi)>0\}$. Then $\mu$ is unimodal if and only if, for any $a>0$, the equation $p(\xi)=a$ has at most two solutions $\xi\in \mathbb T$.  
\end{lemma}
\begin{proof}
Assume that $\mu$ is unimodal with mode $\phi$ and antimode $\psi$. Since $p$ is continuous we have $0<\phi-\psi <2\pi$. We identify the function $p$ on $\mathbb{T}$ with a function on $[\psi,\psi+2\pi)$. Since $p$ is real analytic wherever it is positive, the function $p$ is strictly increasing on $\{\theta \in (\psi,\phi): p(\theta)>0\}$ and strictly decreasing on $\{\theta\in (\phi,\psi+2\pi): p(\theta)>0\}$. 

\begin{enumerate}[{Case} 1.]
\item 

 If $a>p(\phi)$ or $0<a<p(\psi)$, then the equation $p(\theta)=a$ has no solutions $\theta\in [\psi,\psi+2\pi)$. 

\item If $a=p(\phi)$, then the equation $p(\theta)=a$ has only one solution, since $p(\phi)$ is the strict maximum.

\item If $a=p(\psi)>0$, then the equation $p(\theta)=a$ has only one solution since $p(\psi)$ is the strict minimum. 

\item If $p(\psi)<a< p(\phi)$, then by the intermediate value theorem and the monotonicity the equation $p(\theta)=a$ has a unique solution in each of $[\psi,\phi)$ and $(\phi,\psi+2\pi)$. 
\end{enumerate}
In any case, we conclude that the equation $p(\theta)=a$ has at most two solutions $\theta$ in $[\psi,\psi+2\pi)$. 

Conversely, we assume that $\mu$ is not unimodal. Let $\psi$ and $\phi$ ($0< \phi - \psi < 2\pi$) be points where $p$ takes a global minimum and a (strict) global maximum, respectively. Then either $p$ is not  non-decreasing on $(\psi, \phi)$ or not non-increasing on $(\phi, \psi+2\pi)$. Without loss of generality we may focus on the former case. Then there exist $\psi < \theta_1 < \theta_2 <\phi$ such that $p(\psi)< p(\theta_2) <p(\theta_1)<p(\phi)$. For any $a \in (p(\theta_2), p(\theta_1))$, by the intermediate theorem the equation $p(\theta)=a$ has at least one solution in each of the intervals $(\psi,\theta_1)$, $(\theta_1,\theta_2)$ and $(\theta_2,\phi)$. 
Therefore the equation $p(\theta)=a$ has at least three solutions. 
\end{proof}

Combining Zhong's density formula and Lemma \ref{lem:basic_unimodal} one can characterized the unimodality for $\mu\boxtimes \lambda_t$ as follows. 

\begin{lemma}\label{lem:unimodal}
Suppose that $t>0$ and $\mu$ is a probability measure on $\mathbb T$. The following statements are equivalent. 
\begin{enumerate}[(1)]
\item $\mu \boxtimes \lambda_t$ is unimodal. 
\item\label{cond-unim2} For any $r\in(0,1)$ the equation 
\begin{equation}\label{sol1}
\frac{1-r^2}{-2\log r}\int_{[-\pi,\pi)} \frac{1}{|1-re^{i(x+\theta)}|^2}\,d\mu(x)=\frac{1}{t}
\end{equation} 
has at most two solutions $\theta \in [-\pi,\pi)$. 
\end{enumerate}
\end{lemma}
\begin{proof}
This proof is very similar to \cite[Proposition 3.8]{HT} or \cite[Lemma 3.1]{HU18}. Recall that $\mu\boxtimes \lambda_t$ is Haar absolutely continuous, its density function $p_t$ is continuous on $\mathbb T$ and real analytic in $\overline{\Psi_{t,\mu}(U_{t,\mu})}$ by Section 2.3. Assume firstly the condition (2). Since $\Psi_{t,\mu}$ is a homeomorphism of $\mathbb{T}$, it suffices to show that for any $a>0$ the equation
\begin{equation}
a=p_t(\overline{\Psi_{t,\mu}(e^{i\theta})})=-\frac{\log(v_{t,\mu}(\theta))}{\pi t}, \qquad \theta\in [-\pi,\pi)
\end{equation}
has at most two solutions $\theta\in [-\pi,\pi)$ by Lemma \ref{lem:basic_unimodal}.  Since $v_{t,\mu}(\theta)$ is a unique solution of the equation \eqref{eq:v} when it is positive, then we have
\begin{equation}\label{solsp}
\left\{ \theta\in [-\pi,\pi): a=-\frac{\log(v_{t,\mu}(\theta))}{\pi t} \right\}=\left\{ \theta\in [-\pi,\pi): \frac{1-(e^{-\pi a t})^2}{-2\log(e^{-\pi a t})}\int _{[-\pi,\pi)}\frac{d\mu(x)}{|1-e^{-\pi a t}e^{i(x+\theta)}|^2}=\frac{1}{t}\right\}.
\end{equation}
By the assumption (2) the equation $ a=-\frac{\log(v_{t,\mu}(\theta))}{\pi t} $ has at most two solutions $\theta \in [-\pi,\pi)$. Conversely, assume that the condition (2) does not hold. Then for some $r\in(0,1)$ there are three distinct solutions $\theta_1,\theta_2,\theta_3\in [-\pi,\pi)$ to the equation \eqref{sol1}. By \eqref{solsp} this shows that $v_{t,\mu}(\theta_k)=r$ and hence $p_t(\overline{\Psi_{t,\mu}(e^{i\theta_k})})=-\frac{\log r}{\pi t}$ for $k=1,2,3$. 
Therefore $\mu\boxtimes \lambda_t$ is not unimodal by Lemma \ref{lem:basic_unimodal} again. 
\end{proof}

\section{Main results}\label{sec5}

In Sections \ref{sec5.1} -- \ref{sec5.3}, we prove the main theorem announced as Theorem \ref{U2}, namely:  in Section \ref{sec5.1}, we study unimodality for the multiplicative convolution of two symmetric probability measures on $\mathbb{T}$; in Section \ref{sec5.2}, we discuss when $\PK_{r,\psi}\circledast \mu$ is unimodal for sufficiently small $r\in(0,1)$ and when it is not unimodal at any time $t>0$; in Section \ref{sec5.3}, we prove that $\PK_{r,\psi}$ is not strongly unimodal for any $r\in(0,1)$ sufficiently close to $1$. 

In Sections \ref{sec5.4} -- \ref{sec5.6}, we prove the main theorem announced as Theorem \ref{U1}, namely: in Sections \ref{sec5.4} and \ref{sec5.5} we discuss when $\mu\boxtimes \lambda_t$ is unimodal for all $t>0$, when it is unimodal for sufficiently large $t>0$, and when it is not unimodal at any $t>0$; in Section \ref{sec5.6}, we conclude the failure of freely strong unimodality for $\lambda_t$ at sufficiently small $t>0$.


\subsection{Multiplicative convolution of symmetric unimodal distributions}\label{sec5.1}

In this section, we study unimodality for the multiplicative convolution of two symmetric unimodal probability measures. It was proved by Wintner that classical additive convolution $\ast$ preserves the symmetric unimodal probability measures on $\mathbb R$ (see \cite[Exercise 29.22]{Sat13} or the original article \cite[Theorem XIII]{W36}). We conjectured a similar property of free additive convolution $\boxplus$ in \cite{HU18}, which is still open.  We now prove the corresponding property for classical multiplicative convolution $\circledast$ on $\mathbb T$. 

\begin{theorem} \label{thm:classical_symmetric_unimodal}
Let $\mu$ and $\nu$ be symmetric unimodal probability measures on $\mathbb T$. Then so is $\mu \circledast \nu$. 
\end{theorem}
\begin{proof}
The symmetry follows from Corollary \ref{Ssym} (one can also prove it more directly by formula \eqref{eq:density_classical_convolution} below). We will prove the unimodality. 
By Lemma \ref{lem:FHS18} and Lemma \ref{wclosed}, we may assume that $\mu$ and $\nu$ respectively have $C^1$ densities $f$ and $g$ on $\mathbb T$. We extend $f$ and $g$  to even continuous functions on $\mathbb R$ which have period $2\pi$, that is, $f(\theta+2\pi)=f(\theta)$ and $g(\theta+2\pi)=g(\theta)$ for all $\theta\in\mathbb{R}$.

 Let $a=-f'$ on $(0,\pi)$. Then $a \geq0$ and  
\begin{equation}
f(\theta) = \int_{|\theta|}^\pi a(\psi)\,d\psi, \qquad \theta \in (-\pi,\pi). 
\end{equation}
Let $h$ be the density of $\mu \circledast \nu$. For $\varphi \in (0,\pi)$ we have 
\begin{align}
h(\varphi)
&= \int_{-\pi}^\pi f(\theta) g(\varphi-\theta) \,d\theta  \label{eq:density_classical_convolution} 
=  \int_{-\pi}^\pi \left(\int_{|\theta|}^\pi a(\psi)\,d\psi \right) g(\varphi-\theta) d\theta 
= \int_{0}^\pi a(\psi)\,d\psi \int_{-\psi}^\psi g(\varphi-\theta)d\theta. 
\end{align}
Fix $\psi \in (0,\pi)$ and let $k(\varphi) :=\int_{-\psi}^\psi g(\varphi-\theta)d\theta=\int_{\varphi-\psi}^{\varphi+\psi}g(\theta)\,d\theta$ for $\varphi \in (0,\pi)$. Then 
\begin{align}
k'(\varphi) = g(\varphi+\psi) - g(\varphi-\psi). 
\end{align}
In order to show $h' \leq0$ on $(0,\pi)$, we show that $k'\leq 0$ on $(0,\pi)$.

Case 1: $\varphi+\psi\leq \pi$. Then $k'\leq 0$ because $\varphi+\psi\geq |\varphi-\psi|$. 

Case 2: $\varphi+\psi \geq \pi$. Then $2\pi-\varphi-\psi \leq \pi$ and 
\begin{align}
k'(\varphi)=g(2\pi-\varphi-\psi)-g(\varphi-\psi),
\end{align}
since $g(\varphi+\psi)=g(-\varphi-\psi)=g(2\pi-\varphi-\psi)$. We can check that $2\pi-\varphi-\psi\geq |\varphi-\psi|$, and hence $g(\varphi-\psi)\geq g(2\pi-\varphi-\psi)$.

Therefore $h$ is non-increasing on $(0,\pi)$ and also non-decreasing on $(-\pi,0)$ since $\mu\circledast \nu$ is symmetric. Hence $\mu\circledast \nu$ is unimodal.
\end{proof}

So far we found no counterexample to the corresponding conjecture for $\boxtimes$. 

\begin{conjecture} Let $\mu$ and $\nu$ be symmetric unimodal distributions on $\mathbb T$. Then $\mu \boxtimes \nu$ is also unimodal. 
\end{conjecture}
Thanks to Proposition \ref{prop:Multiplicative with Poisson} and Theorem  \ref{thm:classical_symmetric_unimodal}, this conjecture holds when one of the distributions is the Poisson kernel. In Section \ref{sec5.4} the above conjecture is proved when one of the probability measures is the free normal distribution.


\subsection{Eventual unimodality and non-unimodality for $\PK_{r,\psi}\circledast\mu$}\label{sec5.2}

Firstly, we give a class of probability distributions $\mu$ on $\mathbb{T}$ such that the distribution $\PK_{r,\psi}\circledast \mu$ is unimodal for sufficiently small $r\in(0,1)$.
\begin{theorem}\label{thm:eventual unimodality of PK}
Let $\varphi \in (0,\pi/2)$ and $\mu$ be a symmetric distribution on $\mathbb{T}$ such that $\text{supp}(\mu)\subset \{e^{i\theta}: \theta\in[-\varphi,\varphi]\}$. Then ${\bf p}_{r,\psi}\circledast\mu $ is unimodal for $r\in(0,r_\varphi)$ where $r_\varphi:=(\cos\varphi)/4$.
\end{theorem}
\begin{proof}
We may assume that $\psi=0$. The density function $f_r$ of ${\bf p}_{r,0}\circledast\mu$ is given by 
\begin{equation}
f_r(\theta)=\frac{1-r^2}{2\pi} \left(\int_0^\varphi \frac{d\mu(x)}{1-2r\cos(\theta-x)+r^2}+\int_0^\varphi\frac{d\mu(x)}{1-2r\cos(\theta+x)+r^2}\right), \qquad \theta \in [-\pi,\pi). 
\end{equation}
For all $\theta\in (0,\pi)$, we have
\begin{equation}
f_r'(\theta)=-\frac{2r(1-r^2)\sin\theta}{\pi} \int_0^\varphi \frac{L_r(\theta,x)}{(1-2r\cos(\theta-x)+r^2)^2(1-2r\cos(\theta+x)+r^2)^2}d\mu(x),
\end{equation}
where
\begin{equation}\label{eq:Lr}
L_r(\theta,x):=\left\{ (1+r^2)^2+4r^2(1-\sin(\theta-x)\sin(\theta+x)) \right\}\cos x-4r(1+r^2)\cos\theta. 
\end{equation}
 If $\theta\in(\pi/2,\pi)$, then it is clear that $L_r(\theta,x)>0$ for all $x\in (0,\varphi)$ and hence $f_r'(\theta)< 0$. If $\theta\in(0,\pi/2)$, then for all $r\in(0,r_\varphi)$ and $x\in(0,\varphi)$, we have
\begin{align}\label{eq:L}
L_r(\theta,x)&\ge (1+r^2)^2\cos x-4r(1+r^2)\cos\theta\\
&\ge (1+r^2)^2\cos x-4r(1+r^2)\\
&\ge (1+r^2)(\cos x-4r)\\
&\ge (1+r^2) (\cos \varphi -4r)\\
\label{eq:L2}&=4(1+r^2)(r_\varphi-r)>0.
\end{align}
Therefore $L_r(\theta,x)>0$ for all $r\in(0,r_\varphi)$ and $x\in (0,\varphi)$. Thus $f_r'(\theta)<0$ for $\theta\in (0,\pi/2)$ if $0<r<r_\varphi$. The above arguments show that $f_r$ is non-increasing on $(0,\pi)$ if $0<r<r_\varphi$. Since ${\bf p}_{r,0}\circledast\mu$ is symmetric, the above result implies that ${\bf p}_{r,0}\circledast\mu$ is unimodal for $r\in(0,r_\varphi)$.
\end{proof}

Next we prove that the above theorem does not hold for $\varphi=\pi/2$. In fact, the equally weighted Bernoulli distribution $\mathbf b$ on $\{1,-1\}\subset\mathbb{T}$ is a counterexample of the above theorem for $\varphi=\pi/2$.

\begin{theorem}\label{thm:b circledast PK}
The measure $\PK_{r,\psi}\circledast\Ber$ attains a strict maximum at the two points $\pm e^{i\psi}$ and hence is not unimodal at any $r\in(0,1)$.
\end{theorem}
\begin{proof}
We may assume that $\psi=0$. The density function of ${\bf p}_{r,0}\circledast\Ber$ is computed in the form
\begin{equation}
f_r(\theta)=\frac{1-r^2}{2\pi} \left(\frac{1}{1-2r\cos\theta+r^2}+\frac{1}{1+2r\cos\theta+r^2}\right), \qquad \theta\in [-\pi,\pi). 
\end{equation}
For all $\theta\in(0,\pi)$, we have
\begin{equation}
f_r'(\theta)=\frac{1-r^2}{2\pi} \cdot \frac{-8r^2(1+r^2)\sin(2\theta)}{(1-2r\cos\theta+r^2)^2(1+2r\cos\theta+r^2)^2}.
\end{equation}
This formula easily shows that the function $f_r$ has a global maximum at $\theta=0,\pi$ and a local minimum at $\theta=\pm\pi/2$. Therefore ${\bf p}_{r,\psi}\circledast\Ber$ is not unimodal at any $r\in(0,1)$.
\end{proof}


\subsection{Failure of strong unimodality of $\PK_{r,\psi}$}\label{sec5.3}
We prove that the Poisson kernel $\PK_{r,\psi}$ is not strongly unimodal for $r \in(0,1)$ sufficiently close to 1.  This is a key fact to prove that the free normal distribution $\lambda_t$ is not freely strongly unimodal for sufficiently small $t>0$ in Section 5.6. We start from defining the concept of classically and freely strong unimodality (we use the concept of freely strong unimodality in Section 5.6 later).

\begin{definition}
A probability measure $\mu$ on $\mathbb{T}$ is said to be {\it strongly unimodal} (resp.\ {\it freely strongly unimodal}) if $\mu \circledast \nu$ (resp.\ $\mu \boxtimes \nu$) is unimodal for every unimodal probability measure $\nu$ on $\mathbb{T}$.
\end{definition}

\begin{theorem}\label{s-unim} There exists $r_0 \in (0,1)$ such that the Poisson kernel $\PK_{r,\psi}$ is not strongly unimodal for any $r \in [r_0,1)$ and any $\psi \in \mathbb R$. 
More strongly, there exists a probability measure $\mu$ and a continuous function $h\colon [r_0,1) \to (0,\infty)$ such that $\limsup_{r\to 1}h(r) \in (0,\infty]$ and for each $r \in [r_0,1)$ and $\psi \in\R$ the equation 
$$
h(r) =\frac{d(\PK_{r,\psi}\circledast \mu)}{d\theta}(\theta)
$$
has at least three solutions $\theta \in [-\pi,\pi)$. 
\end{theorem}

\begin{figure}[h]
\begin{center}
\begin{minipage}{0.45\hsize}
\begin{center}
\includegraphics[width=60mm,clip]{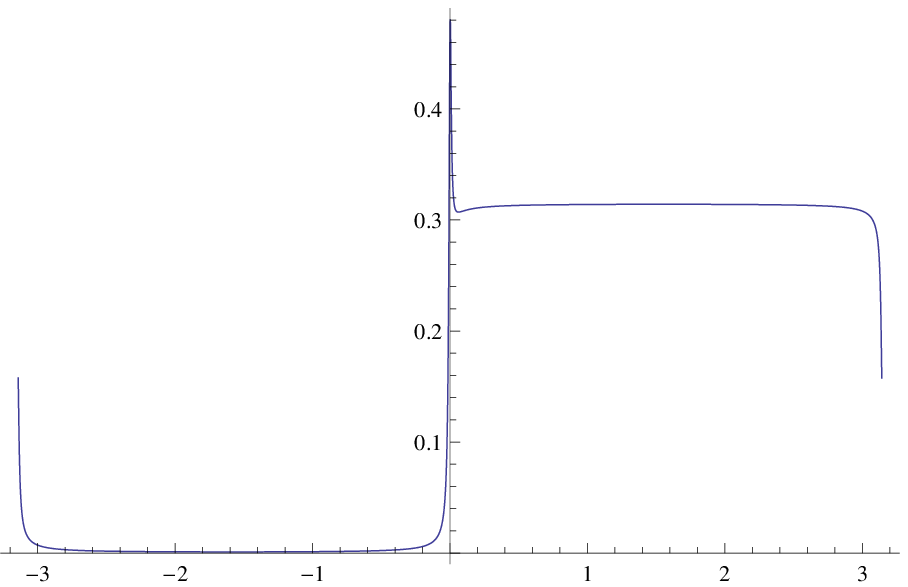}
\end{center}
  \end{minipage}
\begin{minipage}{0.45\hsize}
\begin{center}
\includegraphics[width=65mm,clip]{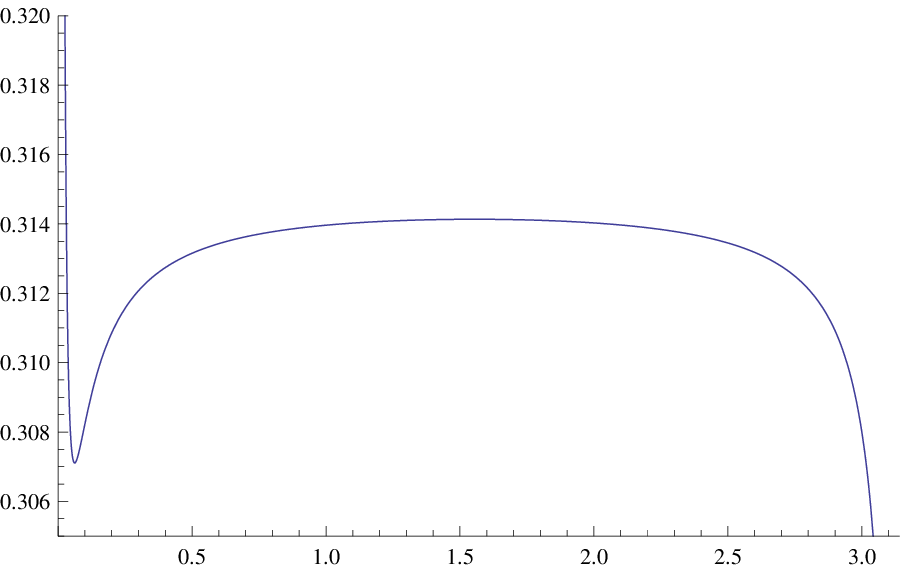}
\end{center}
\end{minipage}
\caption{The density of $\PK_{0.99,0}\circledast \mu_{0.01}$} \label{UFBM1}
\end{center}
\end{figure}

\begin{proof} We may assume that $\psi =0$. Let $\mu=\mu_a$ be the probability measure defined by 
\begin{equation}
\mu_a(d\theta)= a \delta_0 + b {\bf 1}_{(0,\pi)}(\theta) \,d\theta, 
\end{equation}
where $a, b$ are fixed positive real numbers such that $a+b\pi=1$. As we see later, we want $a$ to satisfy $a^2/(4b^2)<1/27$; for example $a=1/20$. Let $f_{r}$ and $p_r$ be the density functions of $\PK_{r,0} \circledast \mu_a$ and $\PK_{r,0}$, respectively. Then we have
\begin{equation}
f_{r}(\theta) =  a p_r(\theta) + b \int_{-\pi}^\pi {\bf 1}_{(0,\pi)}(\theta-\varphi) p_r(\varphi)\,d\varphi = a p_r(\theta) + b\int_{\theta-\pi}^{\theta} p_r(\varphi)\,d\varphi 
\end{equation} 
for all $\theta \in (0,\pi)$. Therefore, we have 
\begin{align}
f_{r}'(\theta) &=  a p_r'(\theta) + b [ p_r(\theta) - p_r(\theta -\pi)] \\
&= \frac{1-r^2}{2\pi} \cdot \frac{g_{r}(\theta)}{(1+r^2 -2r \cos \theta)^2(1+r^2 +2r \cos \theta)} , \qquad \theta \in (0,\pi), 
\end{align} 
where 
\begin{equation}
\begin{split}
g_{r}(\theta) 
&= -2r a \sin \theta (1+r^2 +2r \cos \theta) + b (1+r^2 -2r \cos \theta)(1+r^2 +2r \cos \theta) \\ 
&\quad - b (1+r^2 -2r \cos \theta)^2. 
\end{split}
\end{equation} 
In order to prove that $g_{r}$ has several zeros, we extend the parameter $r$ to the interval $[0,1]$, and then take $r=1$ to simplify $g_{r}$:  
\begin{equation}
g_{1}(\theta) = 4 [-a \sin \theta (1+\cos \theta) +2b \cos \theta (1-\cos \theta)].  
\end{equation}
Note that $g_{1}(0)=0$,  $g'_{1}(0)=-8a<0$ and $g_{1}(\pi)<0$. It is clear that $g_{1} <0$ on $[\pi/2,\pi)$. 
By trigonometry, for $\theta \in (0,\pi/2)$ the inequality $g_{1}(\theta)<0$ is equivalent to 
\begin{equation}
\frac{a^2}{4b^2} > \frac{\cos^2\theta (1-\cos\theta)}{(1+\cos \theta)^3}. 
\end{equation}
The function on the RHS increases on $(0, \pi/3)$ and decreases on $(\pi/3, \pi/2)$, and attains the maximum $1/27$ at $\theta =\pi/3$. 
Thus, if $a^2/(4b^2)<1/27$ (for example $a =1/20$ suffices) then we obtain two solutions $\alpha_{1}, \beta_{1}$ of the equation 
\begin{equation}
\frac{a^2}{4b^2} = \frac{\cos^2\theta (1-\cos\theta)}{(1+\cos \theta)^3},  \qquad \theta \in [0,\pi/2], 
\end{equation}
such that $0<\alpha_{1}< \pi/3 < \beta_{1}<\pi/2$ and $g_{1} >0 $ on $(\alpha_{1},\beta_{1})$ and $g_{1}<0$ on $(0,\alpha_{1}) \cup (\beta_{1}, \pi)$.  Each of the zeros $0,\alpha_{1}, \beta_{1}$ of $g_{1}$ has multiplicity one. 

Since $g_{r}(\theta)$ depends continuously on $r \in [0,1]$ and analytically on $\theta$, by the argument principle, there exists $r_1 \in (0,1)$ such that the function $g_{r}$ still has three zeros $-0.1<\gamma_{r}< \alpha_{r}< \beta_{r}<\pi/2+0.1$ for every $r \in(r_1,1)$.  This implies that $f_{r}$ takes local maxima at $\gamma_{r}$ and $\beta_{r}$ and a local minimum at $\alpha_{r}$. The zeros are continuous functions of $r \in[r_1,1]$. Since $g_{1}'(0)=-8a<0$, $g_{1}(0)=0$ and $g_{r}(0) >0$ for every $r \in [0,1)$, we must have $\gamma_{r}>0$. 

The zero $\beta_{r}$ depends on $r \in [r_1,1]$ continuously, and in particular $\lim_{r\to1}\beta_{r} = \beta_{1} \in (\pi/3,\pi/2)$. Therefore we have $\beta_{r} \in (\pi/3,\pi/2)$ if $r$ is close enough to 1, and hence 
\begin{equation}
b\int_{\beta_{r}-\pi}^{\beta_{r}} p_r(\varphi)\,d\varphi \leq f_{r}(\beta_{r}) \leq a p_r\left(\frac{\pi}{3}\right) + b\int_{\beta_{r}-\pi}^{\beta_{r}} p_r(\varphi)\,d\varphi.  
\end{equation}
Taking the limit $r\to 1$ we conclude that 
\begin{equation}
\lim_{r\to1} f_{r}(\beta_{r}) = b.  
\end{equation}
On the other hand, the local maximum at $\gamma_{r}$ can be estimated from below as 
\begin{equation}
f_{r}(\gamma_{r})\geq f_{r}(0) \geq a p_r(0) = \frac{a(1+r)}{2\pi(1-r)}. 
\end{equation}
Therefore, if $r$ is close enough to $1$ then $f_{r}(\gamma_{r})> f_{r}(\beta_{r})$, and so the equation $f_{r}(\theta) = h$ has at least three solutions $\theta \in (0,\pi)$ for every $h \in (f_{r}(\alpha_{r}), f_{r}(\beta_{r}))$ (see Figure \ref{UFBM1}). Thus 
there exist $r_0 \in (0,1)$ and a continuous function $h\colon [r_0,1)\to (0,\infty)$ such that the equation $f_{r}(\theta) =h(r)$ has at least three solutions $\theta \in (0,\pi)$ for each $r$, and $h(r)\to b$ as $r\to1$; for example we may define $h(r) = \max \{\frac{1}{2}[f_{r}(\alpha_{r})+ f_{r}(\beta_{r})], f_{r}(\beta_{r})-(1-r)\}$.   
\end{proof}

\begin{remark}
Numerical simulations suggest that $\PK_{0.9,0} \circledast \mu_a$ is already unimodal for all $a \in (0,1)$; see Figures \ref{FBM2} and \ref{FBM3}. 
\end{remark}
\begin{figure}[h]
\begin{center}
\begin{minipage}{0.45\hsize}
\begin{center}
\includegraphics[width=65mm,clip]{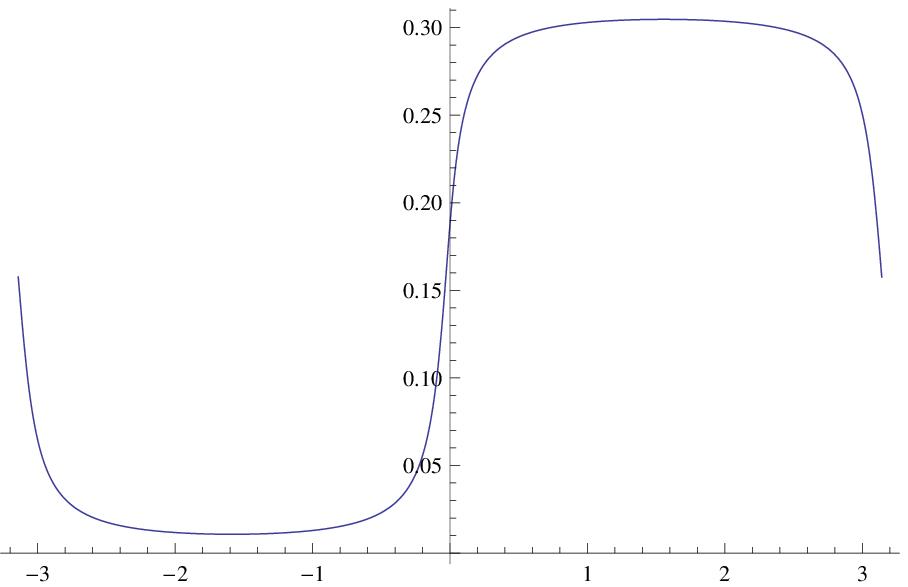}
\end{center}
\caption{The density of $\PK_{0.9,0}\circledast \mu_{0.01}$} \label{FBM2}
  \end{minipage}
\begin{minipage}{0.45\hsize}
\begin{center}
\includegraphics[width=65mm,clip]{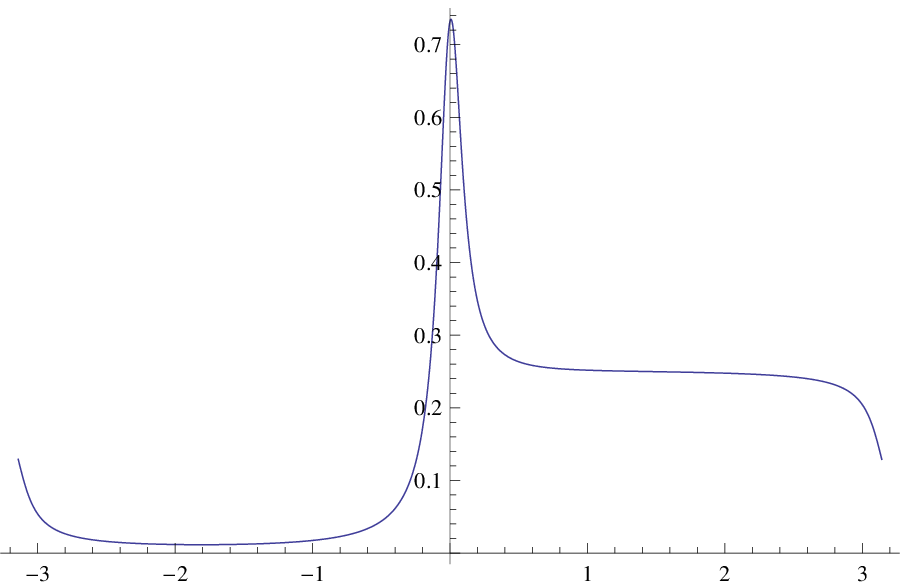}
\end{center}
\caption{The density of $\PK_{0.9,0}\circledast \mu_{0.2}$} \label{FBM3}
\end{minipage}
\end{center}
\end{figure}

We do not know whether $\PK_{r,0}$ is strongly unimodal for $r$ smaller than $r_0$.

\begin{problem}
Does there exist $r \in (0,1)$ such that the Poisson kernel $\PK_{r,0}$ is strongly unimodal? Note that $\PK_{0,0}$ is the normalized Haar measure and the weak limit $\lim_{r\uparrow1}\PK_{r,0}$ is the delta measure $\delta_1$, and hence both are strongly unimodal. 
\end{problem}


A more challenging problem is to prove the analogue of Ibragimov's theorem, which characterizes the strongly unimodal distributions on $\R$ by the log concavity of the density (see \cite[Theorem 52.3]{Sat13}; the original article is \cite{I56}).

\begin{problem}
Characterize the classically strong unimodality on $\mathbb T$.  In particular, is the classical normal distribution (= the heat kernel) on $\mathbb T$  strongly unimodal? 
\end{problem}


\subsection{Unimodality for $\mu\boxtimes \lambda_t$}\label{sec5.4}

In this section, we study unimodality for $\mu\boxtimes \lambda_t$ when $\mu$ is symmetric unimodal. 

\begin{theorem}\label{main-t}
If $\mu$ is a symmetric unimodal probability measure on $\mathbb{T}$, then $\mu \boxtimes \lambda_t$ is symmetric and unimodal for any $t>0$.
\end{theorem}

\begin{proof}
Since $\mu$ and $\lambda_t$ are symmetric, the probability measure $\mu\boxtimes \lambda_t$ is also symmetric by Corollary \ref{Ssym}. To show the unimodality, it suffices by Lemma \ref{lem:unimodal} to show that for any $r\in(0,1)$ and $t>0$ the equation
\begin{equation}\label{equation:unimodal}
\frac{1-r^2}{2\pi}\int_{[-\pi,\pi)} \frac{1}{|1-re^{i(\theta+x)}|^2}d\mu(x)=\frac{-\log r}{\pi t}
\end{equation}
has at most two solutions $\theta\in [-\pi,\pi)$. Since $\mu$ is symmetric, the LHS of \eqref{equation:unimodal} corresponds to the density function of $\PK_{r,0} \circledast \mu$. By Theorem \ref{thm:classical_symmetric_unimodal}, the measure $\PK_{r,0}\circledast \mu$ is unimodal for any $r\in(0,1)$, and by Lemma \ref{lem:basic_unimodal} the equation \eqref{equation:unimodal} has at most two solutions $\theta\in[-\pi,\pi)$ for any $r\in(0,1)$ and $t>0$ . Hence $\mu\boxtimes \lambda_t$ is unimodal for any $t>0$.
\end{proof}


\subsection{Eventual unimodality and non-unimodality for $\mu\boxtimes \lambda_t$}\label{sec5.5}
We firstly prepare the following functions to describe the density of $\mu \boxtimes \lambda_t$. For $\theta\in[-\pi,\pi)$ and $r\in(0,1)$ define the following two functions on $[-\pi,\pi)$. 
\begin{align}
\Gamma_{\theta,r}(x):=&\frac{\sin(\theta+x)}{(1-2r\cos(\theta+x)+r^2)^2},\\
\Xi_{\theta,r}(x):=&\frac{2r(1+\log r-r^2+r^2\log r)\cos(\theta+x)+r^4-4r^2\log r-1}{(1-2r\cos(\theta+x)+r^2)^2}.
\end{align}
Those functions will play a key role in analyzing the density of $\mu\boxtimes \lambda_t$ as the following lemma shows. Actually, it serves as an alternative of Lemma \ref{lem:unimodal}; one can choose a convenient one from the two lemmas, according to a problem to be considered. 
\begin{lemma}\label{lem:derivative_v}
For arbitrary numbers $\theta \in [-\pi, \pi)$ and $r\in (0,1)$,  the function $\Xi_{\theta,r}$ takes negative values on $[-\pi,\pi)$. Moreover, for each $t>0$ we have
\begin{equation}
\frac{d}{d\theta}v_{t,\mu}(\theta)=\frac{2v_{t,\mu}(\theta)^2(1-v_{t,\mu}(\theta)^2)\log(v_{t,\mu}(\theta))}{\int_{-\pi}^\pi\Xi_{\theta,v_{t,\mu}(\theta)}(x) \,\mu(dx)} \int_{-\pi}^\pi \Gamma_{\theta,v_{t,\mu}(\theta)}(x)\,\mu(dx), \qquad \theta \in U_{t,\mu}. 
\end{equation}
\end{lemma}

\begin{proof} 
By calculus we can check that $1+\log r-r^2+r^2\log r<0$ for all $r\in (0,1)$. Hence we have
\begin{align}
&2r(1+\log r-r^2+r^2\log r)\cos(\theta+x)+r^4-4r^2\log r-1 \\
&\qquad \le -2r(1+\log r-r^2+r^2\log r)+r^4-4r^2\log r-1 \\
&\qquad = (1+r)^2 (-2r \log r +r^2 -1). 
\end{align}
By calculus we can prove that the RHS is negative, and so is $\Xi_{\theta,r}(x)$. 

Recall that 
the function $v_{t,\mu}$ is analytic in $U_{t,\mu}$ and satisfies $v_{t,\mu}(\theta) \in (0,1)$ and the equation $R(\theta,v_{t,\mu}(\theta))=0$  for all $\theta \in U_{t,\mu}$, where 
\begin{equation}\label{eq2}
R(\theta,r):=\frac{1-r^2}{-2\log r} \int_{-\pi}^\pi \frac{\mu(dx)}{1-2 r\cos(\theta+x)+r^2}-\frac{1}{t},\qquad r\in(0,1), \theta \in U_{t,\mu}. 
\end{equation}
In order to compute $\frac{d}{d\theta} v_{t,\mu}(\theta)$, we calculate partial derivatives of $R$ which can be expressed with the functions $\Gamma_{\theta,r}$ and $\Xi_{\theta,r}$ as follows:
\begin{equation}
\begin{split}
\frac{\partial R}{\partial \theta}(\theta,r)=&\frac{r(1-r^2)}{\log r}\int_{-\pi}^\pi \Gamma_{\theta,r}(x)\,\mu(dx),
\end{split}
\end{equation}
and
\begin{equation}
\begin{split}
\frac{\partial R}{\partial r}(\theta,r)= -\frac{1}{2r(\log r)^2} \int_{-\pi}^\pi \Xi_{\theta,r}(x)\,\mu(dx). 
\end{split}
\end{equation}
Since $\frac{\partial R}{\partial r}(\theta,r)>0$, we can use the differentiation of implicit functions to obtain the desired expression of $v_{t,\mu}'$. 
\end{proof}

\begin{lemma} \label{lem:estimate_v}
Let $t>4$. Let $r_t \in (0,1)$ be the unique solution $r$ to the equation
$$
\frac{1+r}{1-r} (-2\log r) = t. 
$$  
Then $v_{t,\mu}(\theta) \leq r_t$ for every $\theta \in [-\pi,\pi)$ and every probability measure $\mu$ on $\mathbb T$.  In particular, $U_{t,\mu}=\mathbb T$ and $\lambda_t \boxtimes \mu$ has a strictly positive real analytic density on $\mathbb T$. 
\end{lemma}

\begin{remark}
The opposite inequality of the form $v_{t,\mu}(\theta) \geq a_t$ was obtained in \cite[Proposition 3.10]{Z15} in a similar way. 
\end{remark}

\begin{proof} By the triangle inequality we obtain 
\begin{equation}\label{eq:inequality_v}
\frac{1-r^2}{-2\log r} \int_{\mathbb T} \frac{d\mu(\xi)}{|1-re^{i\theta} \xi|^2} \geq \frac{1-r^2}{-2\log r} \cdot \frac{1}{(1+r)^2} =  \frac{1-r}{(1+r)(-2\log r)} =: f(r). 
\end{equation}
Elementary calculus shows that $f$ is strictly increasing on $(0,1)$ and $f(+0)=0$ and $f(1-0) = 1/4$. Therefore, for every $t >4$ the unique solution $r_t$ exists. Since $f(r_t)=1/t$, the LHS of \eqref{eq:inequality_v} is greater than or equal to $1/t$ at $r=r_t$, which implies that $v_{t,\mu}(\theta) \leq r_t$. Since $v_{t,\mu}(\theta)<1$ the last statement holds. 
\end{proof}

We give a class of probability distributions $\mu$ on $\mathbb{T}$ such that the distribution $\mu\boxtimes \lambda_t$ is unimodal for sufficiently large time. In fact the class considered in Theorem \ref{thm:eventual unimodality of PK} is available. 

\begin{theorem}\label{main-t1}
Let $\varphi \in (0,\pi/2)$ and $\mu$ be a symmetric probability measure on $\mathbb T$ such that $\supp(\mu) \subset \{e^{i\theta}: \theta \in [-\varphi,\varphi]\}$. Then $\mu \boxtimes\lambda_t$ is unimodal for all $t\geq \frac{2(1+r_\varphi)}{1-r_\varphi} \log \frac1{r_\varphi}$, where $r_\varphi = (\cos\varphi)/4$. 
\end{theorem}

\begin{proof} 
Let $t>4$. For $\theta\in \mathbb T=U_{t,\mu}$, we have
\begin{align}
\int_{-\pi}^\pi &\Gamma_{\theta,v_{t,\mu}(\theta)}(x) \,d\mu(x)=\int_{0}^\varphi \left[\Gamma_{\theta,v_{t,\mu}(\theta)}(x)+ \Gamma_{\theta,v_{t,\mu}(\theta)}(-x) \right]d\mu(x) \\
&=2\sin\theta\int_{0}^\varphi  \frac{ L_{v_{t,\mu}(\theta)} (\theta,x)\,d\mu(x)}{(1-2v_{t,\mu}(\theta)\cos(\theta+x)+v_{t,\mu}(\theta)^2)^2(1-2v_{t,\mu}(\theta)\cos(\theta-x)+v_{t,\mu}(\theta)^2)^2},
\end{align}
where
the function $L_r(\theta,x)$ was defined in \eqref{eq:Lr}. For $\theta\in [0,\pi]$,  we have 
\begin{equation}
\begin{split}
L_{v_{t,\mu}(\theta)}(\theta,x)
\ge  4(1+v_{t,\mu}(\theta)^2)(r_\varphi-v_{t,\mu}(\theta)).
\end{split}
\end{equation}
by the same calculations from \eqref{eq:L} to \eqref{eq:L2}. Suppose that $r_\varphi \geq  v_{t,\mu}(\theta)$ holds for all $\theta \in [0,\pi]$. We can then conclude $L_{v_{t,\mu}(\theta)}(\theta,x) \geq 0$ for all $\theta \in [0,\pi]$ and $x \in [0,\varphi]$, and hence $\int_{-\pi}^\pi \Gamma_{\theta, v_{t,\mu}(\theta)}(x) \,d\mu(x)\geq 0$. By Lemma \ref{lem:derivative_v}, we have $\frac{d}{d\theta}v_{t,\mu}(\theta)\geq  0$ for $\theta\in[0,\pi]$, and therefore $\frac{d}{d\theta}v_{t,\mu}(\theta)\leq  0$ for $\theta\in[-\pi,0]$ since $\mu$ is symmetric. Thus $v_{t,\mu}$ is non-increasing in $(-\pi,0)$ and non-decreasing in $(0,\pi)$. Therefore the density function $p_t(\overline{\psi_t(e^{i\theta})})=-\frac{\log(v_{t,\mu}(\theta))}{\pi t}$ of $\mu \boxtimes \lambda_t$ is non-decreasing in $(-\pi,0)$ and non-increasing in $(0,\pi)$, so that $\mu \boxtimes \lambda_t$ is unimodal. 

Next we prove $r_\varphi\geq v_{t,\mu}(\theta)$ for large $t$ by using Lemma \ref{lem:estimate_v}. Let $t_\varphi:= \frac{1+r_\varphi}{1-r_\varphi} (-2\log r_\varphi)$, which is greater than $4$.  For every $t \geq t_\varphi$ we have $r_t \leq r_\varphi$ since $t\mapsto r_t$ is strictly decreasing,  and hence $r_\varphi\geq v_{t,\mu}(\theta)$.  
\end{proof}

Next we prove that the above theorem does not hold for $\varphi=\pi/2$. This result is a complete free analogue of Theorem \ref{thm:b circledast PK} on the classical multiplicative convolution of the Poisson kernel $\PK_{r,\psi}$ and the Bernoulli distribution.
\begin{theorem}\label{main-t2}
The measure $\Ber \boxtimes \lambda_t$ attains a strict maximum at $\pm1 \in \mathbb T$ and hence is not unimodal at any time $t>0$.
\end{theorem}
\begin{proof}
By the definition of $U_{t,\Ber}$, we have
\begin{equation}
\begin{split}
U_{t,\Ber}&=\left\{ -\pi\le \theta < \pi: -\sqrt{\frac{t}{2}}<\sin\theta<\sqrt{\frac{t}{2}}\right\}.  
\end{split}
\end{equation}
Therefore, $U_{t,\Ber}$ and hence the support of $\Ber \boxtimes \lambda_t$ have two connected components if $0<t<2$, and has the single connected component $\mathbb T$ if $t\geq2$. The density function of $\Ber \boxtimes \lambda_t$ is symmetric for the $x$-axis since $\Ber$ is symmetric.  For $\theta \in U_{t,\Ber}$ we have 
\begin{align}
\int_{[-\pi,\pi)} \Gamma_{\theta,v_{t,\Ber}(\theta)}(x)\,\Ber(dx) 
&= \frac{1}{2} \left(\frac{\sin\theta}{(1-2r \cos \theta +r^2)^2} +\frac{-\sin\theta}{(1+2r \cos \theta +r^2)^2} \right)  \\
&= \frac{2r(1+r^2)\sin(2\theta)}{(1-2r \cos \theta +r^2)^2(1+2r \cos \theta +r^2)^2}. 
\end{align}
By Lemma \ref{lem:derivative_v} we conclude that $\frac{d}{d\theta}v_{t,\Ber}(\theta)>0$ on $(0,\frac{\pi}{2})\cap  U_{t,\Ber}$ and $\frac{d}{d\theta}v_{t,\Ber}(\theta)<0$ on $(\frac{\pi}{2},\pi)\cap U_{t,\Ber}$. Since $v_{t,\Ber}(\theta)=v_{t,\Ber}(-\theta)$, the function $v_{t,\Ber}$ has local minima at $\theta=0,\pi$. By the density formula $p_t(\overline{\Psi_{t,\Ber}(e^{i\theta})})=-\frac{\log(v_{t,\Ber}(\theta))}{\pi t}$,  the desired statement holds for any $t>0$.
\end{proof}

\begin{remark}
We have concluded in \cite{HU18} that the free additive convolution of the semicircle distribution $S(0,t)$ and a compactly supported probability measure on $\mathbb{R}$ is unimodal for sufficiently large $t>0$, but a similar statement is not true in the case of the free multiplicative convolution, according to Theorem \ref{main-t2}.
\end{remark}


\subsection{Failure of freely strong unimodality of $\lambda_t$}\label{sec5.6}
In this section, combining Lemma \ref{lem:unimodal} and Theorem \ref{s-unim} we conclude the failure of freely strong unimodality for $\lambda_t$ for small $t>0$. 

\begin{theorem}\label{main-t3}
There exists some $t_0>0$ such that the free normal distribution $\lambda_{t}$ is not freely strongly unimodal for any $t \in (0,t_0)$. 
\end{theorem}
\begin{proof} In order to prove that $\lambda_t$ is not freely strongly unimodal,  according to Lemma \ref{lem:unimodal} we need to find a unimodal distribution $\mu$ and some $r>0$ such that the equation 
\begin{equation}\label{eq:key}
 \frac{d(\PK_{r,0}\circledast\mu)}{d\theta}(\theta)=\frac{-\log r}{\pi t}
\end{equation}
has at least three solutions $\theta \in[-\pi,\pi)$. We take the probability measure $\mu$ and the function $h\colon [r_0,1)\to(0,\infty)$ in Theorem \ref{s-unim}. Let $t_0:=(-\log r_0)/(\pi h(r_0)).$ For each $t \in (0,t_0)$, by the intermediate value theorem there exists some $r_t\in[r_0,1)$ such that $-\log r_t/(\pi t) = h(r_t)$. 
For $r=r_t$ equation \eqref{eq:key} has at least three solutions $\theta \in[-\pi,\pi)$ by Theorem \ref{s-unim}. 
\end{proof}

\begin{problem}
In the case of additive free convolution, once we establish that $S(0,t)$ is not freely strongly unimodal at some $t>0$ then nor is it at any $t>0$, because $t>0$ is just a scaling. However, on the unit circle the parameter $t>0$ for $\lambda_t$ is not a scaling. Thus the following problem remains unsolved: does there exist $t >0$ such that the free normal distribution $\lambda_t$ is freely strongly unimodal? 
\end{problem}

We conclude this paper by mentioning some relationship between additive and multiplicative free convolutions. 
Our method of studying unimodality for free multiplicative convolution has been quite similar to the case of free additive convolution in \cite{HU18}. In particular, the study of the unimodality of $\lambda_t \boxtimes \mu$ was reduced to that of $\PK_{r,0}\circledast \mu$ in this paper, while in \cite[Lemma 3.1]{HU18} the study of the unimodality of $S(0,t)\boxplus \mu$ was reduced to that of the classical convolution of $\mu$ and Cauchy distributions. Actually this kind of similarity between additive and multiplicative free convolutions were observed in several situations, for example in \cite{AH13}.  Anshelevich and Arizmendi \cite{AA} succeeded in systematically explaining those similarities: they proved that various results on additive free convolution can be transferred to multiplicative free convolution on $\mathbb T$ using the exponential mapping. However, unimodality seems out of the applicability of their approach. The main difficulty is twofold: the exponential mapping does not preserve unimodality; the key class $\mathcal L$ of probability measures on $\R$, defined in \cite{AA}, contains only few unimodal distributions (Cauchy and delta measures).


\subsection*{Acknowledgment}
T.H.\ is supported by JSPS Grant-in-Aid for Young Scientists (B) 15K17549 and for Scientific Research (B) 18H01115. The authors are financially supported by JSPS and MAEDI Japan--France Integrated Action Program (SAKURA).



\vspace{0.6cm}

{\it \hspace{-6mm}Takahiro Hasebe\\
Department of Mathematics, Hokkaido University,\\
Kita 10, Nishi 8, Kita-Ku, Sapporo, Hokkaido, 060-0810, Japan\\
email: thasebe@math.sci.hokudai.ac.jp}\\
\vspace{1mm}\\
\hspace{6mm}{\it \hspace{-6mm}Yuki Ueda\\
Department of Mathematics, Hokkaido University,\\
Kita 10, Nishi 8, Kita-Ku, Sapporo, Hokkaido, 060-0810, Japan\\
email: yuuki1114@math.sci.hokudai.ac.jp}


\end{document}